\documentclass[11pt,intlimits]{amsart}

\usepackage{amsmath}
\usepackage{amssymb}
\usepackage{amscd}
\usepackage{amsthm}
\usepackage[all]{xy}
\usepackage{fullpage}
\usepackage{color}
\usepackage{ifthen}
\usepackage{graphicx}
\usepackage{verbatim}

\usepackage[backref]{hyperref}
\usepackage[alphabetic,backrefs,lite,nobysame]{amsrefs}

\theoremstyle{plain}
\newtheorem{theorem}[equation]{Theorem}
\newtheorem{lemma}[equation]{Lemma}

\newtheorem{prop}[equation]{Proposition}
\newtheorem{cor}[equation]{Corollary}

\theoremstyle{definition}

\theoremstyle{remark}
\newtheorem{remark}[equation]{Remark}

\numberwithin{equation}{subsection}

\newcommand\Q{{c}}		
\newcommand\s{{s}}		

\newcommand\SnAQ{{S_{n,\Q,q}(A)}}

\newcommand\dk{{d_{k,\rho}}}
\newcommand\dkLeg{d_{k,\text{Leg}}}

\newcommand\bs\bigskip
\newcommand\ssni{\smallskip\noindent}
\newcommand\np\newpage

\newboolean{usecolor}
\setboolean{usecolor}{true}

\ifx\color@rgb\@undefined\else
	\definecolor{violet}{rgb}{0.25,0,0.75}
	\definecolor{green}{rgb}{0,0.7,0}
\fi

\newcommand{\defi}[1]{{\sl #1}}

\newcommand\bbA{\mathbb{A}}

\newcommand\bbC{\mathbb{C}}
\newcommand\bbE{\mathbb{E}}
\newcommand\bbF{\mathbb{F}}

\newcommand\bbQ{\mathbb{Q}}

\newcommand\bbZ{\mathbb{Z}}

\newcommand\CC{\mathcal{C}}

\newcommand\GG{\mathcal{G}}

\newcommand\II{\mathcal{I}}

\newcommand\MM{\mathcal{M}}
\newcommand\PP{\mathcal{P}}

\newcommand\RR{\mathcal{R}}

\renewcommand\SS{{\mathcal{S}}}

\newcommand\Fq{\bbF_q}
\newcommand\Fpi{\bbF_\pi}
\newcommand\Fv{\bbF_v}

\newcommand\Kbar{\bar{K}}

\newcommand\Qbar{\bar\bbQ}
\newcommand\Qellbar{{\bar\bbQ_\ell}}

\newcommand\KS{K_S}

\newcommand\GammaOf[1]{{\Gamma_q(#1)}}

\newcommand\BQ{\GammaOf{\Q}}
\newcommand\BQFq{{(\bbF_q[t]/\Q\,\bbF_q[t])^\times}}

\newcommand\Frob{{\operatorname{Frob}}}
\newcommand\Gal{{\operatorname{Gal}}}
\newcommand\GL{{\operatorname{GL}}}
\newcommand\Hom{{\operatorname{Hom}}}

\newcommand\Res{\operatorname{Res}}

\newcommand\Tr{{\operatorname{Tr}}}

\newcommand\Var{{\operatorname{Var}}}

\newcommand\geom{{\operatorname{geom}}}

\newcommand\std{{\operatorname{std}}}

\newcommand{\vleq}{\rotatebox[origin=c]{90}{$\leq$}}
\newcommand\ssm{\smallsetminus}
\newcommand\seq{\subseteq}
\newcommand\sub{\subset}

\newcommand\into\hookrightarrow
\newcommand\onto\twoheadrightarrow
\newcommand\longto\longrightarrow

\newcommand\one{\mathbf{1}}

\newcommand\ZeroSpace{0}

\newcommand\GK{{G_K}}
\newcommand\GKC{{G_{K,\CC}}}
\newcommand\GKS{{G_{K,\SS}}}

\newcommand\GKR{{G_{K,\RR}}}
\newcommand\Gv{{G_v}}

\newcommand\LC{{L_\CC}}
\newcommand\Lf{{L_{\text{fin}}}}

\newcommand\Aone{{\bbA^1}}

\newcommand\PhiOf[1]{\Phi_q(#1)}

\newcommand\PhiQ{\PhiOf{\Q}}

\newcommand\PhiU{{\PhiOf{u}}}

\newcommand\PhiUNu{{\PhiU^\nu}}

\newcommand\PhiQGood[1]{{\PhiQ_{#1\,\rhomod{good}}}}

\newcommand\PhiQMixed[1]{{\PhiQ_{#1\,\rhomod{mixed}}}}
\newcommand\PhiQHeavy[1]{{\PhiQ_{#1\,\rhomod{heavy}}}}
\newcommand\PhiQBig[1]{{\PhiQ_{#1\,\rhomod{big}}}}

\newcommand\rhomod[1]{\mathrm{#1}}

\newcommand\trhochi{\theta_{\rho,q,\dc}}

\newcommand\Ggeom[2]{{\GG_\geom(#2,#1\PhiUNu)}}

\newcommand\dc{{\varphi}}

\newcommand\chinot{{\mathbf{1}}}
\newcommand\chiz{{\dc_!}}
\newcommand\chibarz{{\bar\dc_!}}

\newcommand\chione{{\dc_1}}
\newcommand\chitwo{{\dc_2}}
\newcommand\chionez{{\dc_{1!}}}
\newcommand\chitwoz{{\dc_{2!}}}
\newcommand\chionebarz{{\bar\dc_{1!}}}
\newcommand\chitwobarz{{\bar\dc_{2!}}}

\newcommand\rholv{{\rho_v}}

\newcommand\rhochi{{\rho\otimes\dc}}

\newcommand\LCStar{{L^*_\CC}}

\newcommand\ckn[1]{{c_{k,#1,n}}}

\newcommand{\loccit}{{\it loc.\,cit.}}

\newenvironment{enum}%
	{%
	 \begin{enumerate}\setlength{\itemsep}{0.075in}}%
	{\end{enumerate}%
	}

\newcommand{\ES}{\mathcal S} 

\newcommand{\divid}{d} 

\begin{document}


\begin{abstract}
We compute the variances of sums in arithmetic progressions of generalised $k$-divisor functions related to certain $L$-functions in $\mathbb{F}_q[t]$, in the limit as $q\rightarrow\infty$. This is achieved by making use of recently established equidistribution results for the associated Frobenius conjugacy classes.  The variances are thus expressed, when $q\rightarrow\infty$, in terms of matrix integrals, which may be evaluated.  Our results extend those obtained previously in the special case corresponding to the usual  $k$-divisor function, when the $L$-function in question has degree one.  They illustrate the role played by the degree of the $L$-functions; in particular, we find qualitatively new behaviour when the degree exceeds one.  Our calculations apply, for example, to elliptic curves defined over $\mathbb{F}_q[t]$, and we illustrate them by examining in some detail the generalised $k$-divisor functions associated with the Legendre curve. 
\end{abstract}


\title[Variance of sums of divisor functions]%
{Variance of sums in arithmetic progressions of divisor functions associated with higher degree $L$-functions in $\mathbb{F}_q[t]$}

\author{Chris Hall%
\address{Department of Mathematics, Western University, London, ON, Canada, N6A 5B7}}

\author{Jonathan P.~Keating%
\address{School of Mathematics, University of Bristol, Bristol BS8 1TW, UK}}

\author{Edva Roditty-Gershon%
\address{School of Mathematics, University of Bristol, Bristol BS8 1TW, UK}}

\thanks{ We are pleased to acknowledge support under EPSRC Programme Grant EP/K034383/1
LMF: \textit{L}-Functions and Modular Forms.  JPK is also grateful for support from a Royal Society Wolfson Research Merit Award and ERC Advanced Grant 740900 (LogCorRM).}


\allowdisplaybreaks
\maketitle


\section{Introduction}\label{sec:introduction}

For a fixed integer $k\geq 2$ the generalized divisor function $\divid_k(n)$ denotes the number of ways of writing a (positive) integer as a product of $k$ factors:
$$
\divid_k(n)=\sum_{m_{1}\cdot m_{2}\cdots m_{k}=n}1
$$
It is related to the $k$th power of the Riemann zeta-function, defined by
\begin{equation}\label{zeta}
	\zeta(s)
	=
	\sum_{n=1}^\infty \frac{1}{n^s}
\end{equation}
when ${\rm Re}s>1$, in that
\begin{equation}\label{zetak}
	\zeta(s)^k
	=
	\sum_{n=1}^\infty \frac{\divid_k(n)}{n^s}
\end{equation}
for ${\rm Re}s>1$.

Define $\Delta_k(x)$ by 
\begin{equation}
\Delta_k(x):=\sum_{n\leq x}\divid_k(n)-\Res_{s=1}\frac{x^{s}\zeta^{k}(s)}{s}=\sum_{n\leq x}\divid_k(n)-xP_{k-1}(\log x)
\end{equation}
where $P_{k-1}(u)$ is a certain polynomial of degree $k-1$; see, for example, \cite{Titch} Chapter XII.

The mean square of $\Delta_2(x)$ was computed by Cr\'amer \cite{Cramer} for $k=2$, and by Tong \cite{Tong} for $k\geq 3$ (assuming the Riemann Hypothesis (RH) if $k\geq 4$), to be 
\begin{equation}\label{eq:Tong}
	\frac 1X\int_X^{2X} \Delta_k(x)^2 dx
	\sim
	c_k X^{1-\frac 1k},
\end{equation}
for a certain constant $c_k$. Heath-Brown \cite{HB1992} showed that $\Delta_k(x)/x^{\frac 12-\frac 1{2k}}$ has a limiting value distribution (for $k\geq 4$ one needs to assume RH); it is non-Gaussian.

Our main focus will be on  sums of  divisor functions over arithmetic progressions
\begin{equation}
	\ES_{\divid_k}(A)
	=
	\ES_{\divid_k;X;Q}(A)
	=
	\sum_{\substack{n\leq X\\n=A\bmod Q}}
	\divid_k(n)
\end{equation}
For the standard divisor function ($k=2$), it is known that if $Q<X^{2/3-\epsilon}$ then
\begin{equation}
	\ES_{\divid_2}(A)
	=
	\frac{X p_Q(\log X)}{\Phi(Q)} + O(X^{1/3+o(1)})
\end{equation}
for some linear polynomial $p_Q$. This is due to unpublished work of Selberg. For recent work on the asymptotics of sums of $\divid_3$ over arithmetic progressions, see \cite{FKM} and the literature cited therein.

The variance $\Var(\ES_{\divid_2;X;Q})$  of $\ES_{\divid_2}$
has been studied by Motohashi \cite{Motohashi},
Blomer \cite{Blomer}, Lau and Zhao \cite{Lau Zhao}, the result being \cite{Lau Zhao} (we assume $Q$ prime for simplicity):

\begin{enum}
\item If $1\leq Q<X^{1/2+\epsilon}$ then
\begin{equation}
	\Var(\ES_{\divid_2;X;Q}) \ll X^{1/2} + (\frac XQ)^{2/3+\epsilon}.
\end{equation}

\item For $X^{1/2}<Q<X$,
\begin{equation}\label{eq:Lau Zhao}
	\Var(\ES_{\divid_2;X;Q})
	=
	\frac XQ p_3(\log \frac{Q^2}X) + O((\frac XQ)^{5/6}(\log X)^3 )
\end{equation}
where $p_3$ is a polynomial of degree $3$ with positive leading coefficient. 
\end{enum}

\noindent
For $k\geq 3$, Kowalski and Ricotta \cite{KowRic} considered smooth analogues of the divisor sums $\ES_{\divid_k;X;Q}(A)$, and among other things computed the variance for $Q^{k-1/2+\epsilon}<X<Q^{k-\epsilon}$.

In the function field setting, the analogous problem was studied by Rodgers, Rudnick and the second and third authors \cite{KRRR}.  Let $\MM_n\sub\Fq[t]$ be the set of monic polynomials of degree $n$ with coefficients in $\Fq$.  For a polynomial $Q \in F_{q}[x]$ of degree at least $2$, and $A$ co-prime to $Q$, set 
\begin{equation}
	\ES_{\divid_k,n,Q}(A)
	:=
	\sum_{\substack{f\in \mathcal M_n\\f=A\bmod Q}} \divid_k(f) \;,
\end{equation}
Then for $n\leq k(\deg Q-1)$ the variance of this sum is given by
\begin{equation}
	\lim_{q\to \infty} \frac{\Var_Q(\ES_{\divid_k,n,Q})}{  q^n/|Q|}
	=
	I_k(n;\deg Q-1)
\end{equation}
where $I_k(n;\deg Q-1)$ is a certain matrix integral over the unitary group $U_{\deg Q -1}(\bbC)$ --- see \cite{KRRR} and (\ref{matrix integral def}) below.  This integral can be evaluated in a number of ways.  Some of the main results obtained in \cite{KRRR} are set out in Section 5 below.  It follows from these, for example that, for the classical divisor function $\divid=\divid_2$, if $\deg Q\geq 2$ and $n\leq 2(\deg Q-1)$, then
\begin{equation}
	\lim_{q\to \infty} \frac{\Var_Q(\ES_{\divid_2,n,Q})}{  q^n/|Q|} =
	\begin{cases}
		\mathrm{Pol}_3(n), & n\leq \deg Q-1 \\
		\mathrm{Pol}_3(2(\deg Q-1)-n),& \deg Q\leq n\leq 2(\deg Q-1)
	\end{cases}
\end{equation}
where
$$
	\mathrm{Pol}_3(x) = \binom{x+3}{3} = (x+1)(x+2)(x+3)/6.
$$
For general $k$, the variance may be expressed in terms of a lattice-point count and is a piecewise polynomial function of degree $k^2-1$. 

The function-field expressions for the variance of the generalized divisor function lead immediately to conjectures in the standard number-field setting \cite{KRRR}.  The conjecture when $k=3$ has recently been proved by Rodgers and Soundararajan \cite{RodgersSound}. 

Our aim here is to extend the line of research reviewed above to arithmetic functions associated with other $L$-functions in the same way that the generalized divisor function is related to the Riemann zeta-function. Our function field results (which we will soon state below; Theorem \ref{thm:Legendre-application} and Theorem \ref{thm:main-theorem}) can be used to motivate predictions in the number field setting. 
In order to illustrate these predictions, we focus now on two representative examples: 
elliptic curve $L$-functions and the Ramanujan $L$-function.

Let $E/\bbQ$ be an elliptic curve of conductor $N$ defined over $\bbQ$. The associated $L$-function $F(s)$ will be denoted by $L(s,E)$ and is given by
$$
L(s,E)=\prod_{p|N}(1-a_{p}p^{-s})^{-1}\prod_{p\nmid N}(1-a_{p}p^{-s}+p^{-2s+1})^{-1}
$$
where $a_{p}$ is the difference between $p+1$ and the number of points on the reduced curve mod $p$
$$
a_{p}=p+1-\#\tilde{E}(\bbF_{p}).
$$
When $p|N$, $a_{p}$ is either $1,$ $-1$, or $0$. In general, we have the Hasse bound on $a_{p},$ $|a_{p}| < 2\sqrt{p}$, hence the product converges and gives an analytic function for all $Re(s)>3/2.$
The $L$ function of $E$ expands as
$$
L(s,E)=\sum_{n=1}^{\infty}a_{n}n^{-s}
$$
where 
$$
a_{p^{e}}=p^{e}+1-\#\tilde{E}(\bbF_{p^{e}}).
$$
and for $n=\prod_{i=1}^{r}p_{i}^{e_{i}}$ with $p_{1},\ldots,p_{r}$ distinct primes
$$
a_{n}=a_{p_{1}^{e_{1}}}\cdots a_{p_{r}^{e_{r}}}.
$$
For each positive integer $k$ consider the multiplicative function $d_{k,E}:\bbZ\rightarrow \bbZ$ given by
$$
d_{k,E}:=\prod_{n_{1}\cdots n_{k}=n}a_{n_{1}}\cdots a_{n_{k}}.
$$
Note that $d_{k,E}$ gives the coefficients in the Dirichlet series expansion
$$
L(s,E)^{k}=\sum_{n=1}^{\infty}d_{k,E}(n)n^{-s}.
$$

Our results in the function field setting are analogous to computing the variance of the sum of $d_{k,E}$ in arithmetic progressions
Define the sum over arithmetic progressions
$$
S_{x,c,E}(A):=\sum_{\substack{n\leq x \\ n=A \mod c}}d_{k,E}(n).
$$
Our function field result (see Theorem~\ref{thm:main-theorem} ) leads us to predict that for $x^{\epsilon}<c $, $\epsilon >0$, the following holds:
$$
\mathrm{Var}(S_{x,c,E}) \sim \frac{x}{\phi(c)}a_{k}(L(s,E))\gamma_k(\frac{\log x}{2\log c})(2\log c)^{k^{2}-1}.
$$
with
$$
	\gamma_k(c)
	=
	\frac{1}{k!\, G(1+k)^2}
	\int_{[0,1]^k}
		\delta_c(w_1 + \ldots + w_k)
		\prod_{i< j}(w_i-w_j)^2\,
		d^k w,
$$
with $\delta_c(x) =\delta(x-c)$ being the delta distribution translated by $c$, and $G$ is the Barnes $G$-function, so that for positive integers $k$, $G(1+k) = 1!\cdot 2! \cdot 3! \cdots (k-1)!$. Here $a_{k}(L(s,E))$ is an arithmetic factor that can be written explicitly by (2.5.9) in \cite{CFKRS}.

Note that we can detect the degree of the $L$-function in question as the coefficient of $\log c$.

Another example of a degree-two $L$-function is the Ramanujan $L$-function: 
$$
L(s,\tau )=\prod _{p}\left(1-\tau (p)p^{-s}+p^{-2s+11}\right)^{-1} ,
$$
where $\tau$ is the Ramanujan tau function $\tau :\mathbb {N} \to \mathbb {Z} $ defined by the following identity:
$$ 
\sum _{n\geq 1}\tau (n)q^{n}=q\prod _{n\geq 1}(1-q^{n})^{24},
$$
where $q=\exp(2\pi iz)$.
Ramanujan conjectured (and his conjecture was proved by Deligne) that
$|\tau(p)| \leq 2p^{11/2}$ for all primes $p$,
hence the product converges and gives an analytic function for all $Re(s)>13/2.$
The $L$ function associated with $\tau$ expands as
$$
L(s,\tau )=\sum_{n=1}^{\infty}\tau(n)n^{-s}
$$
For each positive integer $k$ consider the multiplicative function $d_{k,\tau}:\bbZ\rightarrow \bbZ$ given by
$$
d_{k,\tau}:=\prod_{n_{1}\cdots n_{k}=n}\tau(n_{1})\cdots\tau(n_{k}).
$$
Note that $d_{k,E}$ gives the coefficients in the Dirichlet series expansion
$$
L(s,\tau )^{k}=\sum_{n=1}^{\infty}d_{k,\tau}(n)n^{-s}.
$$
Again we are led to speculate that for $x^{\epsilon}<c $, $\epsilon>0$,  if  
$$S_{x,c,\tau}(A):=\sum_{\substack{n\leq x \\ n=A \mod c}}d_{k,\tau}(n) $$
then the following holds:
$$
\mathrm{Var}(S_{x,c,\tau}) \sim \frac{x}{\phi(c)}a_{k}(L(s,\tau))\gamma_k(\frac{\log x}{2\log c})(2\log c)^{k^{2}-1}..
$$
Where $a_{k}(L(s,\tau))$ is an arithmetic factor that can be written explicitly by (2.5.9) in \cite{CFKRS}.

We set out our main results below, but first, by way of illustration, we examine a specific example where the $L$-function in question take a relatively simple and explicit form, and can defined in a self-contained way.  


\subsection{Divisors associated to Legendre curve L-function}

Let $q$ be a power of a prime $p$.  Let $\MM\sub\Fq[t]$ be the subset of monic polynomials and $\II\sub\MM$ be the subset of irreducibles.  For each $n\geq 1$, let $\MM_n\sub\MM$ and $\II_n\sub\II$ be the respective subsets of elements of degree $n$.

Suppose $q$ is odd, and let $E/\Fq(t)$ be the Legendre curve, that is, the elliptic curve with affine model
$$
	y^2 = x(x-1)(x-t).
$$

\ssni
Its $L$-function is given by an Euler product
$$
	L(T,E/\Fq(t))
	=
	\prod_{\pi\in\II}
	L(T^{\deg(\pi)},E/\Fpi)^{-1}
$$
where $\Fpi$ is the residue field $\Fq[t]/\pi\,\Fq[t]$.

Each Euler factor of $L(T,E/\Fq(t))$ is the reciprocal of a polynomial in $\mathbb{Q}[T]$ and satisfies
$$
	T\frac{d}{dT}
	\log L(T,E/\mathbb{F}_\pi)^{-1}
	=
	\sum_{m=1}^\infty a_{\pi^m}T^m
	\in\mathbb{Z}[[T]].
$$
We regard the coefficients $a_{\pi^m}$ as values of the multiplicative function $f\mapsto a_f$ on $\MM$ given by the coefficients of the Dirichlet series expansion
$$
	L(T,E/\Fq(t))
	=
	\sum_{f\in\MM}
	a_f\,T^{\deg(f)}.
$$
Thus $a_1=1$, and if $f=\prod_{i=1}^r\pi_i^{e_i}$ with $\pi_1,\ldots,\pi_r$ distinct elements of $\II$, then
$$
	a_f = a_{\pi^{e_1}}\cdots a_{\pi^{e_r}}.
$$

For each positive integer $k$, consider the multiplicative function $\dkLeg\colon\MM\to\bbZ$ given by
$$
	\dkLeg(f)
	=
	\prod_{\substack{f_1,\ldots,f_k\in\MM\\[0.01in]f_1\cdots f_k=f}}
	a_{f_1}\cdots a_{f_k}.
$$
Equivalently, $\dkLeg$ gives the coefficients in the Dirichlet series expansion
$$
	L(T,E/\Fq(t))^k
	=
	\sum_{f\in\MM}\dkLeg(f)\,T^{\deg(f)}.
$$
It is easy to see that this is a generalization of the $k$th divisor function: replace $L(T,E/\Fq(t))$ by the zeta function $Z(T,\Aone/\Fq)$ so that $f\mapsto a_f$ becomes the constant function $f\mapsto 1$.

Let $\Q\in\MM$ be square free and $\BQ=\BQFq$.  For each $n\geq 1$ and $A\in\BQ$, consider the sum
$$
	S_{k,n,\Q}(A)
	\ :=
	\sum_{\substack{f\in\mathcal{M}_n\\f\equiv A\bmod\Q}}
	\dkLeg(f).
$$
Let $A$ vary uniformly over $\Gamma(\Q)$, and consider the expected value
$$
	\bbE_A[S_{k,n,\Q}(A)]
	:=
	\frac{1}{|\BQ|}
	\sum_{A\in\BQ}
	S_{k,n,\Q}(A)
$$
and the variance
$$
	\mathrm{Var}_A[S_{k,n,\Q}(A)]
	:=
	\frac{1}{|\BQ|}
	\sum_{A\in\BQ}
	|S_{k,n,\Q}(A)-\bbE_A[S_{k,n,\Q}(A)]|^2.
$$
These moments depend on $q$, so one can ask how they behave when we replace $\Fq$ by a finite extension, that is, let $q\to\infty$.  Our main result gives the variance  $\mathrm{Var}_A[S_{k,n,\Q}(A)]$, in the limit $q\to \infty$, in terms of a matrix integral.  Let $U$ be an $R\times R$ matrix.  Let $\std\colon U_R(\bbC)\to\GL_R(\bbC)$ be the representation given by the inclusion $U_R(\bbC)\sub\GL_R(\bbC)$ and
$$
	\wedge^j \std\colon U_R(\bbC)\to \GL_{R_j}(\wedge^j\bbC)
$$
be its $j$th exterior power where $R_j=\binom{R}{j}$; we define $(\wedge^j\,\std)(g)=0$ unless $0\leq j\leq R$.  Define the matrix integrals with respect to Haar measure $d\theta$ on $U_R(\bbC)$ the group of $R\times R$ unitary matrices
\begin{equation}\label{matrix integral def}
	I_{k}(n;R)
	=
	\int_{U_R(\bbC)}
	\Big|
		\sum_{n_1+\cdots +n_k=n}
		\,\,
		\Tr\Big(\big(\otimes_{i=1}^k (\wedge^{n_i}\,\std)\big)(\theta)\Big)
	\Big|^2
	d\theta .
\end{equation}
Then, as a special case of a general theorem we present in this paper, one can prove the following theorem:

\begin{theorem}\label{thm:Legendre-application}
If $\gcd(\Q,t(t-1))=t$ and if $\deg(\Q)\gg 1$, then
$$
	\lim_{q\to\infty}
	\frac{|\BQ|}{q^{2n}}
	\Var_A[S_{k,n,\Q}(A)]
	=
	I_{k}(n;2\deg (c)-1)
$$
\end{theorem}

See Remark~\ref{rmk:proof-of-legendre-theorem} for a brief explanation of how this follows from our general theorem.  The factor 2 multiplying $\deg(c)$ corresponds to the degree of the Legendre curve $L$-function.

\begin{cor}
In the case of $k=2$ we have for $n\leq 2\deg(c)-1$
$$
	\lim_{q\to\infty}
	\frac{|\BQ|}{q^{2n}}
	\Var_A[S_{2,n,\Q}(A)]
	=
	{n+3\choose 3}
$$
and for $2\deg(c)-1 < n < 4\deg(c)-2$
$$
	\lim_{q\to\infty}
	\frac{|\BQ|}{q^{2n}}
	\Var_A[S_{2,n,\Q}(A)]
	=
	{4 \deg c -n+1\choose 3}.
$$
\end{cor}

\ssni
The Corollary will follow from evaluating the matrix integral in \S\ref{integral}.


\section{Statement of Main Theorem}

Let $\Kbar$ be an algebraic closure of $K=\Fq(t)$ and $\GK$ be the Galois group $\Gal(\Kbar/K)$.  We regard $\PP=\II\cup\{\infty\}$ as the set of places of $K$, and for each $v\in\PP$, we write $\Fv$ for the residue field and $\deg(v)$ for its degree over $\Fq$.  We also fix inertia and decomposition groups $I(v)\sub D(v)\sub\GK$ respectively and write $\Gv$ for the quotient group $D(v)/I(v)$ and $\Frob_v\in\Gv$ for the Geometric Frobenius element.

For each finite subset $\SS\sub\PP$, let $\KS\seq\Kbar$ be the maximal subfield which is unramified over $\PP\ssm\SS$ and $\GKS$ be the Galois group $\Gal(\KS/K)$.  We abuse notation and write $I(v)\seq D(v)$ for the images of $I(v)\sub D(v)\sub\GK$ in $\GKS$ via the canonical quotient $\GK\onto\GKS$.  If $v\in\SS$, then $I(v)\seq\GKS$ is isomorphic to $I(v)\sub\GK$, and otherwise $I(v)\seq\GKS$ is the trivial group.

Let $\ell$ be a prime distinct from $p$ and $V$ be a finite-dimensional $\Qellbar$-vector space.  Consider a Galois representation
$$
	\rho\colon\GKS\to\GL(V),
$$

\ssni
that is, a continuous group homomorphism.  For each $v\in\PP$, let $V_v=V^{\rho(I(v))}$. We say that $s$ is the conductor of $\rho$ if $s$ is a square-free monic polynomial divisible by prime polynomials $v$ for which $V_v$ is strictly smaller than $V.$
Let
$$
	\rholv\colon\Gv\to\GL(V_v)
$$
be the composition of the restriction of $\rho$ to $D(v)$ and the quotient $D(v)\onto\Gv$, and let
$$
	L(T,\rho_v) := \det\left(1 - T\,\rho(\Frob_v)\mid V_v\right).
$$

We attach several $L$-functions to $\rho$.  One is the complete $L$-function and is given by the Euler product
\begin{equation}\label{eqn:euler-product-for-L}
	L(T,\rho) := \prod_{v\in\PP}L(T^{\deg(v)},\rho_v)^{-1}.
\end{equation}
Another is a partial $L$-function given by the Euler product
$$
	\Lf(T,\rho) := \prod_{\pi\in\II}L(T^{\deg(\pi)},\rho_\pi)^{-1}
$$
where now the Euler product is taken over the `finite' places of $K$.  If $L(T,\rho_\infty)$ is trivial, that is, if $V_\infty=\ZeroSpace$, then these $L$-functions coincide, but in general
$$
	\Lf(T,\rho) = L(T,\rho)\cdot L(T,\rho_\infty).
$$

One reason for considering the partial $L$-function $\Lf(T,\rho)$ is that we can expand its Euler product as a Dirichlet series
$$
	\Lf(T,\rho) = \sum_{f\in\MM}a_f T^{\deg(f)}
$$

\ssni
where $f\mapsto a_f$ is the multiplicative function $\MM\to\Qellbar$ given on prime powers by writing
$$
	T\frac{d}{dT}\log(L(T,\rho_\pi)^{-1})
	=
	\sum_{m=1}^\infty a_{\pi^m}T^m.
$$
Equivalently, for each $\pi\in\II$ and $m\geq 1$, we have
$$
	a_{\pi^m} = \Tr(\rho_v(\Frob_\pi)^m\mid V_v);
$$
and if $f=\prod_{i=1}^r\pi_i^{e_i}$ is a prime factorization in $\Fq[t]$ with $\pi_i\neq\pi_j$ for $i\neq j$, then
$$
	a_f = a_{\pi_1^{e_1}}\cdots a_{\pi_r^{e_r}}
$$
where $a_1=1$.

For each positive integer $k$, we define the \defi{$k$th divisor function of $\rho$} as follows: it is the multiplicative function $\dk\colon\MM\to\Qellbar$ given by
$$
	\dk(f)
	=
	\prod_{\substack{f_1,\ldots,f_k\in\MM\\[0.01in]f_1\cdots f_k=f}}
	a_{f_1}\cdots a_{f_k}.
$$
When $\rho$ is the trivial representation $\one$ (and thus $\dim(V)=1$), then this is the usual $k$ divisor function on $\MM$:
$$
	d_{k,\one}(f)
	=
	\left|\left\{\,
		f_1,\ldots,f_k\in\MM
		:
		f_1\cdots f_k=f
	\,\right\}\right|
$$

\ssni
In general, we have the identity
$$
	\Lf(T,\rho)^k
	=
	\sum_{f\in\MM} \dk(f)\,T^{\deg(f)}
	=
	\sum_{n=0}^\infty\left(\sum_{f\in\MM_n}\dk(f)\right)T^n.
$$

Let $\CC\sub\PP$ be a finite subset containing $\infty$ and
\begin{equation}\label{eqn:euler-product-of-L_C}
	\LC(T,\rho) := \prod_{\pi\not\in\CC}L(T^{\deg(\pi)},\rho_\pi)^{-1}
\end{equation}
where $\pi$ runs over $\II$ and let 
\begin{equation}
R=\deg(\LC(T,\rho)).
\end{equation}
Thus $\Lf(T,\rho)=L_{\{\infty\}}(T,\rho)$, and in general,
\begin{equation}\label{eqn:L_CC-to-d_k}
	\LC(T,\rho)^{k}
	=
	\sum_{n=0}^\infty\Big(
		\sum_{\substack{f\in\MM_n\\[0.02in]\gcd(f,\Q)=1}}\dk(f)
	\Big)T^n
\end{equation}
where $\Q\in\MM$ is the product of all primes $\pi\in\II\cap\CC$.

Let $\BQ$ be the finite abelian group $\BQFq$, and for each $A\in\BQ$, let
$$
	\SnAQ := \sum_{f\in\MM_n(A)}\dk(f)
$$
where
$$
	\MM_n(A) := \{\,f\in\MM_n : f\equiv A\bmod\Q\,\}.
$$
We regard $A$ as a random variable uniformly distributed over $\BQ$ and define its expected value
$$
	\bbE_A[\SnAQ]
	:=
	\frac{1}{|\BQ|}\sum_{A\in\BQ}\SnAQ
$$
and variance
$$
	\Var_A[\SnAQ]
	:=
	\frac{1}{|\BQ|}\sum_{A\in\BQ}\left|\SnAQ-\bbE_A[\SnAQ]\right|^2
$$
accordingly.

It follows easily the definition and \eqref{eqn:L_CC-to-d_k} that
$$
	\LC(T,\rho)
	=
	|\BQ|\cdot
	\sum_{n=0}^\infty
	\bbE_A[\SnAQ]\,T^n,
$$
so each expected value is the coefficient of an $L$-series.  The main goal of this paper is to analyze the asymptotic behavior of $|\Var_A[\SnAQ]|^2$ as $q\to\infty$, that is, as we replace $q$ by a power $q^r$ and take $r\to\infty$.  To do so, we must impose some hypotheses on $\rho$, e.g., we suppose that $\rho$ is punctually pure of weight $w$ (see section $6$ in \cite{HKRG}).  We also impose hypotheses on the Mellin transform of $\rho$, but before doing so we need some additional notation and terminology.

Let $\PhiQ$ be the finite abelian group $\Hom(\BQ,\Qellbar^\times)$.  For each $\dc\in\PhiQ$, there is a corresponding Dirichlet character
$$
	\dc\colon\GKC\to\GL(\Qellbar)
$$

\ssni
which we regard as a representation of $\GKR$ for $\RR=\CC\cup\SS$ by composing $\dc$ with the quotient $\GKR\onto\GKC$.  We also regard $\rho$ as a representation of $\GKR$ via the quotient $\GKR\onto\GKS$, and we define the tensor-product representation
$$
	\rhochi\colon\GKR\to\GL(V_\varphi)
$$
where $V_\varphi=V$ and $f\mapsto\rho(f)\varphi(f)$.

Let $L(T,\rhochi)$ and $\LC(T,\rhochi)$ be the $L$-functions respectively defined by the Euler products in \eqref{eqn:euler-product-for-L} and \eqref{eqn:euler-product-of-L_C} with $\rhochi$ in lieu of $\dc$.  A priori each, of these is a power series with coefficients in $\Qellbar$, but Grothendieck showed both are rational functions in $\Qellbar(T)$, compare (1.4.7) of \cite{Deligne:WeilII}.  We say that $\dc\in\PhiQ$ is \defi{good for $\rho$} iff it lies in the set
\begin{equation}\label{good}
\PhiQGood\rho
	:=
	\{\,
		\dc\in\PhiQ
		:
		L(T,\rhochi) = \LC(T,\rhochi)\in\Qellbar[T]
	\,\},
\end{equation}
and otherwise we say that $\rho$ is \defi{bad for $\rho$}.

The hypothesis that $\rho$ is punctually pure implies that both $L$-functions lie in $\Qbar(T)$, and $\dc$ is good for $\rho$ iff every zero $\alpha\in\Qbar$ of $\LC(T,\rhochi)$ satisfies $|\iota(\alpha)|^2=1/q^{1+w}$ for every field embedding $\iota\colon\Qbar\to\bbC$.  Equivalently, $\dc$ is good for $\rho$ iff the `unitarized' $L$-function
$$
	\LCStar(T,\rhochi) := \LC(T/(\sqrt{q})^{1+w},\rhochi)
$$

\ssni
is the characteristic polynomial of a unitary conjugacy class $\trhochi\sub U_R(\bbC)$ for $R=\deg(\LC(T,\rho))$.

We further distinguish bad characters by saying that $\rho\in\PhiQ$ is \defi{mixed for $\rho$} iff it lies in the set
\begin{equation}\label{mixed}
\PhiQMixed\rho
	:=
	\{\,
		\dc\in\PhiQ\ssm\PhiQGood\rho
		:
		\LC(T,\rhochi)\in\Qellbar[T]
	\,\},
\end{equation}
and otherwise we say that the elements of
\begin{equation}\label{heavy}
	\PhiQHeavy\rho
	:=
	\PhiQ\ssm(\PhiQGood\rho\cup\PhiQMixed\rho)
\end{equation}
are \defi{heavy for $\rho$}.  The mixed characters are those for which $\LCStar(T,\rhochi)$ is not the characteristic polynomial of a unitary matrix, and the heavy characters are those for which $\LCStar(T,\rhochi)$ is not even a polynomial.



We are now in a position to state our main theorem:

\begin{theorem}\label{thm:main-theorem}
Suppose that $\rho$ is punctually pure of weight $w$, that its Mellin transform has big monodromy (see section \ref{Big Monodromy}), and that $\PhiQHeavy\rho\seq\{\one\}$ for all $q$. Then 
$$
	\lim_{q\to\infty}\frac{|\PhiQ|}{q^{n(1+w)}}\cdot \Var_A[S_{k,n,\Q}(A)]
	=
	I_{k}(n;R)
$$
for each $n\geq 1$.
\end{theorem}

\begin{remark}\label{rmk:proof-of-legendre-theorem}
If $\rho$ is the representation associated to the $\ell$-adic Tate module of the Legendre curve, then the hypotheses on $\Q$ in Theorem~\ref{thm:Legendre-application} imply that the Mellin transform of $\rho$ has big monodromy (see \cite[\S8]{HKRG}).  One can also show that $R=2\deg(\Q)-1$ (cf.~\loccit).  Thus Theorem~\ref{thm:Legendre-application} follows from Theorem~\ref{thm:main-theorem}.
\end{remark}

\begin{cor}
In the case of $k=2$ we have for $n\leq R$
$$
	\lim_{q\to\infty}
	\frac{|\BQ|}{q^{2n}}
	\Var_A[S_{2,n,\Q}(A)]
	=
	{n+3\choose 3}
$$
and for $R<n<2R$
$$
	\lim_{q\to\infty}
	\frac{|\BQ|}{q^{2n}}
	\Var_A[S_{2,n,\Q}(A)]
	=
	{2R-n+3\choose 3}
$$
\end{cor}


\section{Big Monodromy and Equidistribution}\label{Big Monodromy}


In \cite[\S10]{HKRG}, we defined a subgroup $\PhiUNu\seq\PhiQ$, and for each coset $\dc\PhiUNu$, we defined a monodromy group
$$
	\Ggeom{\rho}{\dc\PhiUNu}
	\seq
	\GL_R(\Qellbar)
$$
generated by Frobenii corresponding to the good characters in $\dc\PhiUNu$, where $R:=\deg(\LC(T,\rho))=\deg(\LC(T,\rhochi))$ (see \cite[Prop.~4.3.1]{HKRG}). More precisely, the Frobenius attached to a character $\dc\alpha^\nu\in\dc\PhiUNu$ is an element of $\GL_R(\Qellbar)$ with characteristic polynomial $\LC(T,\rho\otimes\chi\alpha^\nu)$, and $\Ggeom{\rho}{\dc\PhiUNu}$ is the Zariski closure of all such elements when one takes $q\to\infty$.

We say that $\dc\in\PhiQ$ is \defi{big for $\rho$} iff it lies in the set
$$
	\PhiQBig\rho
	:=
	\{\,
		\dc\in\PhiQ
		:
		\Ggeom{\rho}{\dc\PhiUNu}
		=
		\GL_R(\Qellbar)
	\,\},
$$
and we say that \defi{the Mellin transform of $\rho$ has big monodromy} iff
$$
	|\PhiQ|
	\sim
	|\PhiQBig\rho|
	\text{ as }
	q\to\infty.
$$

\begin{lemma}\label{size}
If the Mellin transform of $\rho$ has big monodromy, then
$$
	|\PhiQ|
	\sim
	|\PhiQGood\rho|
	\sim
	|\PhiQBig\rho|
	\text{ as }
	q\to\infty
$$
\end{lemma}

\begin{proof}
See \cite[Corollary .~10.4.3]{HKRG}.
\end{proof}
If the Mellin transform of $\rho$ has big monodromy it implies that $\trhochi$ become equidistributed in $U_R(\bbC)$ (for reference see Theorem 10.0.4 combined with remark 8.2.4 in \cite{HKRG}):
\begin{theorem}\label{equidistribution}
Suppose that $\rho$ is punctually pure of weight $w$, that its Mellin transform has big monodromy, and that $\PhiQHeavy\rho\seq\{\one\}$ for all $q$. Then for any continuous function $F:U_R(\bbC)\rightarrow \bbC$
\begin{equation}
	\lim_{q\to\infty}
	\frac{1}{|\PhiQGood\rho|}
	\sum_{\dc\in\PhiQGood\rho}
	F(\trhochi)
	\ =
	\int_{U_R(\bbC)}
	F(\theta)
	d\theta
\end{equation}
with respect to Haar measure $d\theta$ on $U_R(\bbC)$.
\end{theorem}
In the following Theorem we present a sufficient criteria for the Mellin transform of $\rho$ to have big monodromy (explicit example for representations meeting this criteria can be found in \cite[\S 12]{HKRG}).
\begin{theorem}
Let $s$ be the conductor of $\rho$ and suppose that $\gcd(\s,c)=t$ and that $\deg c\geq 3$.  Suppose moreover that $V(0)$ has a unique unipotent block of exact multiplicity one and that $\rho$ is geometrically simple and pointwise pure.  If $r:=\dim(V)$ and $\deg c$ satisfy
\begin{equation}
\deg c >\frac{1}{r}(72(r^2+1)^2 - r - \deg(L(T,\rho)) +\sum_{\nu \in C}\deg \nu(r-\deg L(T,\rho_{\nu
}))
\end{equation}
then the Mellin transform of $\rho$ has big monodromy.
\end{theorem}




\section{Proof of Theorem \ref{thm:main-theorem}}

We express the variance of the arithmetic progressions sums $S_{k,n,c}(A)$ in terms of sums of divisor functions twisted by Dirichlet characters.  Let $k$ be a positive integer and let
\begin{equation}\label{twisted-d-sum}
	\ckn{\rhochi}
	=
	\sum_{\substack{f\in\MM_n\\[0.02in]\gcd(f,\Q)=1}} d_{k,\rhochi}(f)
\end{equation}
be the coefficients in the expansion of the k-th power of the partial L-function \eqref{eqn:L_CC-to-d_k}
$$
	\LC(T,\rhochi)^k
	=
	\sum_{n=0}^\infty
	\ckn{\rhochi}
	\,T^n
$$
For each $\dc\in\PhiQ$, we extend $\dc$ to a multiplicative map $\chiz\colon\MM\to\Qellbar$ by defining:
$$
	\chiz(f)
	=
	\begin{cases}
		\dc(f+\Q\,\Fq[t]) & \mbox{if }
		\gcd(f,\Q)=1 \\
		0 & \mbox{otherwise}.
	\end{cases}
$$
It is multiplicative and satisfies
\begin{equation*}\label{eq:chis.vs.chiz}
	\chiz(\pi)
	=
	\begin{cases}
		\dc(\Frob_{v(\pi)}) &  \mbox{if }\pi\nmid\Q \\
		0 & \mbox{otherwise}
	\end{cases}
	\mbox{ for }\pi\in\II.
\end{equation*}

The Orthogonality relations for Dirichlet characters are:
\begin{enum}
\item for each $A_1,A_2\in\BQ$
\begin{equation}\label{eq:orth:A}
	\frac{1}{|\PhiQ|}\sum_{\dc\in\PhiQ}\chiz(A_1)\chibarz(A_2)
	=
	\begin{cases}
	1 & \mbox{if }A_1=A_2 \mbox{ mod } c \\
	0 & \mbox{if }otherwise,
	\end{cases}
\end{equation}
\item for every $\chione,\chitwo\in\PhiQ$
\begin{equation}\label{eq:orth:chi}
	\frac{1}{|\PhiQ|}\sum_{A\in\BQ}\chionez(A)\chitwobarz(A)
	=
	\begin{cases}
	1 & \mbox{if }\chione=\chitwo \\
	0 & \mbox{if }\chione\neq\chitwo.
	\end{cases}
\end{equation}
\end{enum}

We use \eqref{eq:orth:A} and \eqref{twisted-d-sum} to express the sum of $d_{k,\rho}(f)$ over arithmetic progressions in terms of $\ckn{\rhochi}:$
\begin{equation*}\label{eqn:sum-of-bs}
	\SnAQ
	=
	\frac{1}{|\PhiQ|}
	\sum_{f\in\MM_n}
		d_{k,\rho}(f)
	\sum_{\dc\in\PhiQ}
		\chiz(A)
		\bar{\chiz}(f)
	=
	\frac{1}{|\PhiQ|}
	\sum_{\dc\in\PhiQ}
		\ckn{\rhochi}\cdot\bar{\chiz}(A) 
\end{equation*}
Therefore, if we write $\chinot\in\PhiQ$ for the trivial character and by using the second orthogonality relation \eqref{eq:orth:chi} then expected value of $\SnAQ$ equals
\begin{equation}
	\begin{split}
    	\bbE_A[\SnAQ]
    	& := \frac{1}{|\PhiQ|}\sum_{A\in\BQ}\SnAQ \\
    	& = \frac{1}{|\PhiQ|^2}
    		\sum_{\dc\in\PhiQ}
        	\ckn{\rhochi}
        	\sum_{A\in\BQ}
    	    	\bar\chiz(A) \\
        & = \frac{1}{|\PhiQ|} \ckn{\rho\otimes\chinot}
	\end{split}
\end{equation}
	
In particular, we have the identity
\begin{equation}\label{eq:s-e}
	\SnAQ - \bbE_A[\SnAQ]
	=
	\frac{1}{|\PhiQ|}
	\sum_{\substack{\dc\in\PhiQ \\ \dc\neq\chinot}}
		\ckn{\rhochi}\cdot\bar\dc(A).
\end{equation}

Now consider the variance
$$
	\Var_A[\SnAQ]
	=
	\frac{1}{|\PhiQ|}
	\sum_{A\in\BQ}
		\left|\SnAQ - \bbE_A[\SnAQ]\right|^2.
$$
If we apply identities \eqref{eq:orth:chi} and \eqref{eq:s-e}, then the right side equals
$$
	\frac{1}{|\PhiQ|^3}
	\sum_{A\in\BQ}
	\sum_{\substack{\chione,\chitwo\in\PhiQ\\\chionez,\chitwoz\neq\chinot}}
		\ckn{\rho\otimes\chione}\overline{\ckn{\rho\otimes\chitwo}}
		\cdot
		\chionebarz(A)\chitwoz(A)
	=
	\frac{1}{|\PhiQ|^2}
	\sum_{\substack{\dc\in\PhiQ\\\dc\neq\chinot}}
		|\ckn{\rhochi}|^2.
$$

\medskip
In summary, the function $\SnAQ$ of the random variable $A$ satisfies
\begin{equation}\label{eqn:E-and-V-formulae}
	\bbE_A[\SnAQ]
	=
	\frac{1}{|\PhiQ|}\ckn{\rho\otimes \chinot},
	\quad
	\Var_A[\SnAQ]
	= 
	\frac{1}{|\PhiQ|^2}
	\sum_{\substack{\dc\in\PhiQ\\\dc\neq\chinot}}
		|\ckn{\rhochi}|^2.
\end{equation}

Next, we break $\ckn\rhochi$ into smaller pieces which we will express in terms of the associated Frobenius matrices:

\begin{lemma}\label{coefficient}
$$
	\ckn\rhochi
	=
	\sum_{n_1+\cdots+n_k=n}
	\,\,
	\prod_{i=1}^k
	\,
	c_{1,\rhochi,n_i}
$$
\end{lemma}

\begin{proof}
This follows immediately from the definition of $\ckn\rhochi$ as the coefficients in the expansion of $\LC(T,\rhochi)^k$:
\begin{equation}
	\LC(T,\rhochi)^k
	=
	(\sum_{n=0}^{\infty}c_{1,\rhochi,n}T^{n})^{k}
	=
	\sum_{n=0}^{\infty}\sum_{n_{1}+\ldots+n_{k}=n}
	\prod_{i=1}^k
		c_{1,\rhochi,n_{i}}T^{n}
\end{equation}
\end{proof}


Recall that $\std\colon U_R(\bbC)\to\GL_R(\bbC)$ is the representation given by the inclusion $U_R(\bbC)\sub\GL_R(\bbC)$ and
$$
	\wedge^j \std\colon U_R(\bbC)\to \GL_{R_j}(\wedge^j\bbC)
$$
is its $j$th exterior power where $R_j=\binom{R}{j}$; and we defined $(\wedge^j\,\std)(g)=0$ unless $0\leq j\leq R$.


\begin{lemma}\label{eval coeff good}
Let $\dc\in\PhiQGood\rho$, then
$$
	c_{1,\rhochi,n}
	=
	(-1)^n
	\cdot
	(\sqrt{q})^{n(1+w)}
	\cdot
	\Tr((\wedge^n\std)(\trhochi))
$$
\end{lemma}

\begin{proof}
A Dirichlet character $\dc$ is good for $\rho$ iff the `unitarized' $L$-function
$$
	\LCStar(T,\rhochi) := \LC(T/(\sqrt{q})^{1+w},\rhochi)
$$

\ssni
is the characteristic polynomial of a unitary conjugacy class $\trhochi\sub U_R(\bbC)$ for $R=\deg(\LC(T,\rho))$.

The coefficients of the characteristic polynomial of an $N\times N$ matrix with eigenvalues $\lambda_{1},\ldots, \lambda_{n}$ are the elementary symmetric functions $\sum_{1\leq i_{1}<\ldots<i_{r}}\lambda_{i_{1}}\cdots \lambda_{i_{r}}$ which give the character of the exterior power representation. Thus we may write
$$
	\LCStar(T,\rhochi)
	= 
	\sum_{i=0}^{R}
		(-1)^i
		\cdot
		(\sqrt{q})^{i(1+w)}
		\cdot
		\Tr((\wedge^i\std)(\trhochi)) T^{i}  
$$       
\end{proof}

\begin{lemma}\label{eval coeff}
If $\dc\in\PhiQ$, then
$$
	|c_{1,\rhochi,n}|^2
	=
	\begin{cases}
    	O\left(q^{n(2+w)}\right) & \text{ if }\dc\in\PhiQHeavy\rho \\[0.05in]
    	O\left(q^{n(1+w)}\right) & \text{otherwise}
	\end{cases}
$$
The implied constants depend only on $\rho$ and $n$.
\end{lemma}

\begin{proof}
Fix a field embedding $\iota\colon\Qbar\to\bbC$ and identify $\Qbar$ with its image.  Recall that, for some integer $s$ depending on $\dc$ and satisfying $s\leq\dim(V)$, we can express $\LC(T,\rhochi)$ as a ratio 
$$
	\prod_{i=1}^{s+R}(1-\alpha_iT)
	\,/
	\prod_{j=1}^s(1-\beta_jT)
$$
where the $\alpha_i$ and $\beta_j$ lie in $\Qbar$ and satisfy
$$
	|\alpha_i|^2\leq q^{1+w},
	\quad
	|\beta_j|^2\leq q^{2+w}.
$$
For reference see equation $(3.4.2)$ and Theorem $6.2.1$ in \cite{HKRG}.
If $\dc\not\in\PhiQHeavy\rho$, then $s=0$ and $|\alpha_i|^2\leq q^{1+w}$, and thus
$$
	|c_{1,\rhochi,n}|^2
	=
	\left|
    	\sum_{i_1<\cdots<i_n}
    	\alpha_{i_1}\cdots\alpha_{i_n}
	\right|^2
	\leq
 	\sum_{i_1<\cdots<i_n}
	\left|
    	\alpha_{i_1}\cdots\alpha_{i_n}
	\right|^2
	\leq
	\binom{R}{k}q^{1+w}.
$$
When $\rho$ is heavy we get
\begin{eqnarray*}
	|c_{1,\rhochi,n}|^2
	& = &
	\left|
		\sum_{n_1+n_2=n}
		(-1)^{n_1}
    	\sum_{i_1<\cdots<i_{n_1}}
    	\alpha_{i_1}\cdots\alpha_{i_{n_1}}
     	\sum_{j_1\leq\cdots\leq\,j_{n_2}}
       	\beta_{j_1}\cdots\beta_{j_{n_2}}
	\right|^2 \\[0.1in]
	& \leq &
	\sum_{n_1+n_2=n}\,
   	\sum_{i_1<\cdots<i_{n_1}}\,
   	\sum_{j_1\leq\cdots\leq\,j_{n_2}}
	\left|
    	\alpha_{i_1}\cdots\alpha_{i_{n_1}}
       	\beta_{j_1}\cdots\beta_{j_{n_2}}
	\right|^2 \\[0.05in]
	& = &
	O(q^{n(2+w)}).
\end{eqnarray*}
\end{proof}

\begin{cor}\label{eval coeff k}
If $\dc\in\PhiQ$, then
$$
	|\ckn\rhochi|^2
	=
	\begin{cases}
    	O\left(q^{n(2+w)}\right) & \text{ if }\dc\in\PhiQHeavy\rho \\[0.05in]
    	O\left(q^{n(1+w)}\right) & \text{otherwise}
	\end{cases}
$$
The implied constants depend only on $\rho$ and $n$.
\end{cor}

\begin{proof}
Combine Lemma \ref{coefficient} and Lemma \ref{eval coeff}.
\end{proof}

\begin{lemma}\label{just good}
If the Mellin transform of $\rho$ has big monodromy and if $\PhiQHeavy\rho\seq\{\one\}$, then
$$
	\sum_{\substack{\dc\in\PhiQ\\[0.02in]\varphi\neq\one}}
	|\ckn\rhochi|^2
	\ \ \sim
	\sum_{\substack{\dc\in\PhiQGood\rho\\[0.02in]\varphi\neq\one}}
	|\ckn\rhochi|^2
	\ \ 
	\text{ as }
	q\to\infty
$$
\end{lemma}

\begin{proof}
We break the sum over all Dirichlet characters mod $c$ to sums over "good", "heavy" and "mixed" characters (see definitions \eqref{good}, \eqref{heavy}, \eqref{mixed})
$$
	\sum_{\substack{\dc\in\PhiQ\\[0.02in]\varphi\neq\one}}
		|\ckn\rhochi|^2
	\ \ =
	\sum_{\substack{\dc\in\PhiQGood\rho\\[0.02in]\varphi\neq\one}}
		|\ckn\rhochi|^2
	+
	\sum_{\substack{\dc\in\PhiQHeavy\rho\\[0.02in]\varphi\neq\one}}
		|\ckn\rhochi|^2
	+
	\sum_{\substack{\dc\in\PhiQMixed\rho\\[0.02in] \varphi\neq\one}}
		|\ckn\rhochi|^2	
$$
which give by Corollary \ref{eval coeff k}
$$
	\sum_{\substack{\dc\in\PhiQGood\rho\\[0.02in]\varphi\neq\one}}
		|\ckn\rhochi|^2
	+
	|\PhiQHeavy\rho\ssm\{1\}|\cdot O(q^{n(2+w)})
	+
	|\PhiQMixed\rho\ssm\{1\}|\cdot O(q^{n(1+w)})
$$
Lemma~\ref{size} and having $\PhiQHeavy\rho\seq\{\one\}$ conclude the proof.
\end{proof}

\begin{prop}
If the Mellin transform of $\rho$ has big monodromy and if $\PhiQHeavy\rho\seq\{\one\}$, then
$$
	\lim_{q\to\infty}
	\frac{|\PhiQ|}{q^{n(1+w)}}
	\cdot
	\Var_A[\SnAQ]
	\ =
	\int_{U_R(\bbC)}
	\left|\sum_{\substack{n_1+\cdots +n_k=n\\0\leq n_1,\ldots,n_k\leq R}}\,\,
		\Tr\Big(\big(\otimes_{i=1}^k (\wedge^{n_i}\,\std)\big)(\theta)\Big)
	\right|^2
	d\theta.
$$
\end{prop}

\begin{proof}
In \eqref{eqn:E-and-V-formulae} we found the following expression for the variance 
$$
	\Var_A[\SnAQ]
	= 
	\frac{1}{|\PhiQ|^2}
	\sum_{\substack{\dc\in\PhiQ\\\dc\neq\chinot}}
		|\ckn{\rhochi}|^2.
$$	
In the limit of $q\rightarrow\infty$ we have from Lemma \ref{just good}	
$$
	\Var_A[\SnAQ]
	\sim
	\frac{1}{|\PhiQ|^2}
	\sum_{\substack{\dc\in\PhiQGood\rho\\\dc\neq\chinot}}
		|\ckn{\rhochi}|^2
$$	
which Lemma \ref{coefficient} and Lemma \ref{eval coeff good} equals to 
$$
	\frac{1}{|\PhiQ|^2}
	\sum_{\substack{\dc\in\PhiQGood\rho\\\dc\neq\chinot}}
    	\left|
    		\sum_{\substack{n_1+\cdots +n_k=n\\0\leq n_1,\ldots,n_k\leq R}}
				\Tr\Big(
					\big(\otimes_{i=1}^k (\wedge^{n_i}\,\std)\big)(\trhochi)
				\Big)
    	\right|^2
$$
Now note that $|\PhiQ|\sim|\PhiQGood\rho|$ (Lemma~\ref{size}) and apply the equidistribution result Theorem \ref{equidistribution} to conclude the proof.
\end{proof}


\section{Matrix Integral}\label{integral}

In this section we state a few results evaluating the matrix integral in certain ranges and its asymptotic behaviour.  Proofs can be found in \cite{KRRR}.

In certain ranges the matrix integral evaluates to a very simple expression. For
$n<R$ or $(k-1)R<n<kR$, we obtain the following formulae:
\begin{theorem}\label{Thm: Ik(m,N) for large m}
Let $I_k(n;R)$ be the matrix integral defined in \eqref{matrix integral def}.  Then
\begin{enum}
\item for $(k-1)R<n<kR$,
$$
I_k(n;R) = \binom{kR-n+k^2-1}{k^2-1}\;.
$$
\item for $n<R$ 
$$
I_k(n;R) = \binom{n+k^2-1}{k^2-1}\;.
$$
\end{enum}
\end{theorem}

There is also a closed-form formula for the matrix integral for any range of the parameters, in terms of a lattice point count:

\begin{theorem}
\label{lattice_count}
$I_k(m;N)$ is equal to the count of lattice points $x= (x_i^{(j)})\in\mathbb{Z}^{k^2}$ satisfying each of the following relations:
\begin{enum}
\item $0 \leq x_i^{(j)} \leq N$ for all $1\leq i,j \leq k$;
\item $x_1^{(k)} + x_2^{(k-1)} + \cdots + x_k^{(1)} = kN - m$, and
\item $x$ is a $k \times k$ matrix whose entries satisfy the following system of inequalities,
$$
    \begin{matrix}
    	x_1^{(1)} & \leq & x_1^{(2)} & \leq & \cdots & \leq & x_1^{(k)} \\
    	  \vleq   &      &   \vleq   &      &        &      & \vleq \\
        x_2^{(1)} & \leq & x_2^{(2)} & \leq & \cdots & \leq & x_2^{(k)} \\
    	  \vleq   &      &   \vleq   &      &        &      & \vleq \\
          \vdots  &      &   \vdots  &      & \ddots &      & \vdots \\
    	  \vleq   &      &   \vleq   &      &        &      & \vleq \\
        x_k^{(1)} & \leq & x_k^{(2)} & \leq & \cdots & \leq & x_k^{(k)}
    \end{matrix}
$$
\end{enum}
\end{theorem}

The asymptotic behavior of $I_k(n;R)$ for $n\approx R$ is given in the following theorem: 
\begin{theorem}\label{thm:Asymp of I}
Let $c := n/R$.  Then for $c \in [0,k]$,
\begin{equation}\label{thm:poly_approx}
	I_k(n;R)
	=
	\gamma_k(c) R^{k^2-1} + O_k(R^{k^2-2}),
\end{equation}
with
$$
	\gamma_k(c)
	=
	\frac{1}{k!\, G(1+k)^2}
	\int_{[0,1]^k}
		\delta_c(w_1 + \ldots + w_k)
		\prod_{i< j}(w_i-w_j)^2\,
		d^k w,
$$
with $\delta_c(x) =\delta(x-c)$ being the delta distribution translated by $c$, and $G$ is the Barnes $G$-function, so that for positive integers $k$, $G(1+k) = 1!\cdot 2! \cdot 3! \cdots (k-1)!$.
%
\end{theorem}


\bibliographystyle{amsplain.bst}
\bibliography{main}


\end{document}